\documentclass[12pt]{article}
\usepackage{amsmath,amssymb,dsfont,tikz-cd,amsthm}

\title{Borel Colouring Bad Sequences}
\author{Keegan Dasilva Barbosa}
\date{ }

\newcommand{\QQ}{\mathbb{Q}}

\newtheorem{theorem}{Theorem}[section]
\newtheorem{prop}{Proposition}
\newtheorem{corollary}{Corollary}[theorem]
\newtheorem{lemma}[theorem]{Lemma}
\theoremstyle{definition}
\newtheorem{definition}{Definition}[section]
\newtheorem{fact}{Fact}[section]

\begin{document}
	\maketitle
	\begin{abstract}
		Every better quasi-order codifies a Borel graph that does not contain a copy of the shift graph. It is known that there is a better quasi-order that codes a Borel graph with infinite Borel chromatic number, though one has yet to be explicitly constructed. In this paper, we show that examples cannot be constructed via standard methods. Moreover, we show that most of the known better quasi-orders are non-examples, suggesting there is still a class of better quasi-orders with interesting combinatorial properties who's elements/members still remain unknown.
	\end{abstract}
	\section{Introduction}
	In \cite{KST}, it was discovered that there is a graph with uncountable Borel chromatic number which is minimal with respect to Borel graph homomorphism. It was then questioned whether or not the shift graph was minimal in this regard in the class of Borel graphs generated by a single Borel function. Interestingly, Pequignot showed in \cite{Peq} that the shift graph contained a Borel subgraph with infinite Borel chromatic number, but could not embed the shift graph in a Borel manner. These graphs were codified by better quasi-orders (BQO). Further work was done by Todorcevic and Vidnyanszky, and it was further proved that no finite basis could exist, let alone a singular minimal element \cite{Vid}. However, there is a deep question that still remains unanswered in \cite{Peq}. While it is known a BQO can codify a graph with infinite Borel chromatic number, by the nature of the proof relying on a complexity argument of Marcone \cite{Marcone2},a concrete BQO has yet to be identified or explicitly constructed. There were some conjectures made by Pequignot. In this paper, we show that these BQOs fail to codify a graph with infinite Borel chromatic number. Moreover, we show that a BQO recursively constructed from simpler BQOs by the classical means of labeling trees or linear orders \cite{Laver1} \cite{Laver2} \cite{Pouzet} also fails to produce graphs with infinite Borel chromatic number. As a consequence, even any countable collection of $\sigma$-scattered linear orders under the embedding relation fails to be complex enough to code a graph with infinite Borel chromatic number.  This suggests there is still a wealth of BQOs with strong indecomposability properties that have yet to be explicitly constructed. \\
	\\
	This paper will be split into two sections excluding introduction and acknowledgements. Section 2 will be focused on the necessary background information required to understand the problem. This includes the basics of Borel graph combinatorics, BQO theory, and the primary tools and techniques used in these fields. It is suggested a reader skips through the parts of this section they are familiar with. If one wants a deeper understanding of these materials, the author suggests any of the following texts \cite{DescSet} \cite{DescGraph} \cite{Kunen} \cite{FRec}\cite{Marcone1}. However, section 2.5 should not be skipped. It contains important definitions specialized for this problem, as well as a very important lemma. Section 3 is split into three components. Section 3.1 is where we will introduce the colouring algorithm that will allow us to effectively Borel $3$-colour graphs by splitting them into homogeneous pieces that are easy to handle. Section 3.2 will be where we prove that if $Q$ codes a Borel $3$-colourable graph (we call such $Q$ \textit{thin}), then so does $Q^{<\omega}$ under the Higman order \cite{Higman}. Consequently, the same will be true of finite trees labelled by members of $Q$, $\mathcal{FT}_Q $, under the standard orders \cite{Laver1} \cite{Laver2} \cite{Pouzet}. Finally, section 3.3 will be where we jump from finite to infinite, and show that if $Q$ is thin, so is the class of $Q$ labeled $\sigma$-scattered linear orders under $\leqq_{emb}$.
	\section{Notation}
	\subsection{Basics of Descriptive Set Theory}
	\begin{definition}
		We call a topological space $X$ \textit{Polish} if its topology is complete, separable, and metrizable. 
	\end{definition}
Not much on the finer combinatorics of Polish spaces will be relevant for us. We'll mostly be interested in the space of countable sequences of a countable set under the topology of pointwise convergence. However, the reader is highly encouraged to read \cite{DescSet} for more on the subject.
\begin{definition}
	Given a set $A$, $[A]^\omega = \{B\subseteq A : |B| = \omega\}$. 
\end{definition}
\begin{fact}
	Given a countable set $A$, $A^\omega$ under the topology of pointwise convergence is Polish. 
\end{fact}
Note that if $A$ is countable, $[A]^\omega$ is also Polish with topology determined by $X_n \rightarrow X$ if and only if $\bigcap\limits_{n= 1}^\infty \bigcup\limits_{j=n}^\infty X_j = X $. Alternatively, after well ordering $A$, one can identify members from $[A]^\omega$ with strictly increasing sequences in $A^\omega$ with the subspace topology. These definitions are equivalent. 
\begin{fact}
	Given a set $A$, we define $A^{<\omega} $ to be the set of functions with domain $\{0,...,n\}$ for some $n\in \omega$ and codomain $A$. 
\end{fact}
Elements in this set are finite sequences of members from $A$. In fact, considering the members as finite sequences will often be of more use to us than thinking of them as functions.
\begin{definition}
	Given a set $A$, we define the set $[A]^{<\omega} = \{B\subseteq A : |B| < \omega \} $. 
\end{definition}
\begin{definition}
	Given two Polish spaces $V_1$ and $V_2$, we say a function\\ $f:V_1 \rightarrow V_2$ is \textit{Baire class 1} if and only if there is a sequence of continuous functions $f_k:V_1 \rightarrow V_2$, $f_k$ converges to $f$ pointwise.
\end{definition}
All Baire class 1 functions are Borel. Moreover, being Baire class 1 guarantees the preimage of every open set is $G_\delta$ (countable intersection of open sets). One can check \cite{KechLouv} for more on Baire class 1 functions. We will need that these functions are Borel, and not much else. Their main appearance is in the construction of the $\partial^\infty$ operation in section 3.2. 
	\subsection{Borel Graph Combinatorics}
	\begin{definition}
		A \textit{Borel graph} $G$ is a pair $(V,E)$, where $V$ is a Polish space and $E\subseteq [V]^2$ is Borel. 
	\end{definition}
We give $[V]^2$ the topology it inherits when viewed as a subset of $V^2$. Since we are working in the context of descriptive set theory, all properties that we will be discussing from graph theory need to be Borel definable. By this, we mean that we will not be interested in graph homomorphisms, but rather Borel graph homomorphisms. Nor will we work with the standard chromatic number. While it may seem as though this is a small requirement, it is actually rather major. Some graphs that are bipartite in the classic sense may have large Borel chromatic number. For more on this, see \cite{KST}. 
\begin{definition}
	Given Borel graphs $G_1=(V_1,E_1)$ and $G_2=(V_2,E_2)$, a \textit{Borel graph homomorphism} is a Borel map $f: V_1 \rightarrow V_2$ with the property that $\forall x,y \in V_1$, $xE_1 y \Rightarrow f(x) E_2 f(y)$. If such a homomorphism exists, we write $G_1 \preceq_B G_2$. 
\end{definition}
\begin{definition}
	Given a Borel graph $G = (V,E)$, we define the \textit{Borel chromatic number of $G$} to be 
	\begin{align*}
		\chi_B (G) &= \text{min} \{|Y| :Y\text{ Polish},\;  G\preceq_B (Y,[Y]^2)  \}
	\end{align*}
\end{definition}
Note that the requirement that the graph $Y$ in the above definition be Polish means that the only possibilities for $\chi_B(G)$ are $\{1,2,3,4,... ,\aleph_0, 2^{\aleph_0} \}$. We will be studying graphs whose edge set is determined by Borel functions. The range of plausible Borel chromatic numbers these graphs can achieve is finite. In particular, the Borel chromatic number can either be $1$, $2$, $3$, or $\aleph_0$. 
\begin{definition}
	We define $s:2^\omega \rightarrow 2^\omega$ via $\forall k \in \omega, s(x)(k) = x(k+1)$.
\end{definition}
\begin{definition}
	We say a Borel graph $G=(V,E_f)$ is \textit{generated by the Borel function $f$ on $V$} if $f:V\rightarrow V$ Borel and
	\begin{align*}
		\forall x,y \in V,  xE_f y \iff (x\neq y) \text{ and } (f(x) = y \text{ or } f(y) = x)
	\end{align*} 
\end{definition}
	\begin{fact}
		Given a Borel graph $G$ generated by a Borel function $f$ on a Polish space $V$, the following are equivalent.
		\begin{itemize}
			\item $\chi_B(G) \leq 3$
			\item $G\preceq_B (2^\omega ,E_s)$
			\item $\exists A \subseteq X$ Borel and $f$ independent such that $\forall x\in X$ $\exists k\in \omega$, $f^k(x) \in A$ (such a set is called \textit{forward recurrent}) \cite{FRec}
		\end{itemize}
	\end{fact}
The universality property of $(2^\omega , E_s)$ is rather instrumental for our colouring algorithm. The goal of the algorithm is to embed as much as we can into $(2^\omega, E_s) $ in a Borel fashion, then work with the homogeneous remainder. The third property of forward recurrence will also be relevant. Consequently, familiarity with the above fact will be of high importance. Note, not every Borel graph generated by a function needs finite Borel chromatic number. The shift graph is a counterexample. 
\begin{definition}
	We call the graph $([\omega]^\omega, S)$ where $S(A) = A\setminus \text{min} A$ the \textit{shift graph}.
\end{definition}
\begin{fact}[Galvin-Prikry]
	Given a $k\in \omega$ and a Borel colouring $c: [\omega]^\omega \rightarrow \{1,...,k \}$, there is an $A\in[\omega]^\omega$ and $i\in \{1,...,k\}$ such that $[A]^\omega \subseteq c^{-1}(i)$.
\end{fact}
\begin{fact}
	$\chi_B(([\omega]^\omega,S  )) = \aleph_0$.
\end{fact}
\begin{proof}
	Simple application of the Galvin-Prikry theorem. 
\end{proof}
\begin{definition}
	Let $f:X\rightarrow X$ and $g:Y\rightarrow Y$ be functions. We call the mapping $h:X\rightarrow Y$ a \textit{factor map} if $\forall x\in X$, $h(f(x)) = g(h(x)) $.
\end{definition}
\begin{lemma}
	Let $X$ and $Y$ be Polish spaces. Let $f:X\rightarrow X$ and $g:Y\rightarrow Y$ be Borel with no fixed point. If $h: X \rightarrow Y$ is a Borel factor map, then $\chi_B((X,E_f) ) \leq \chi_B((X,E_g)) $.
\end{lemma}
\begin{proof}
	It suffices to show that $(X,E_f) \preceq_B (Y,E_g) $. The factor map $h:X\rightarrow Y $ is a graph homomorphism. To see this, take $x,y \in X$ with $x E_f y$. Since $f$ has no fixed point, either $f(x) =y$ or $f(y) = x$. Without loss of generality, suppose $f(x) =y$. It follows that $h(y) = h(f(x)) = g(h(x)) $. Since $g$ has no fixed point, $h(x) \neq h(y)$ and so it must be the case that $h(x) E_g h(y)$. Hence, $h$ is a Borel graph homomorphism.   
\end{proof}
Most of the graphs we will be interested in will have no fixed point. Note, the shift operation $S$ on $[\omega]^\omega$ has no fixed point. 
	\subsection{Basics of BQO Theory}
	\begin{definition}
		A \textit{quasi-order} is a pair $(Q,\leqq_Q)$ where the relation $\leqq_Q$ on $Q$ is reflexive and transitive. 
	\end{definition}
\begin{definition}
	Given two quasi-orders $Q_1$ and $Q_2$, we assign an order $\leqq_{Q_1\cup Q_2}$ to the disjoint union where $p \leqq_{Q_1\cup Q_2} q \iff \exists i \in \{1,2\}, p,q \in Q_i$ and $p\leqq_{Q_i} q$. 
\end{definition}
\begin{definition}
	We say a quasi-order $(Q,\leqq_Q)$ is a \textit{better-quasi-order} (BQO) if when $Q$ is endowed with the discrete topology, any Borel (equiv. continuous) map $f:[\omega]^\omega \rightarrow Q$, $\exists A\in [\omega]^\omega$, $\forall B\in [A]^\omega$, $f(B) \leqq_Q f(S(B))$. 
\end{definition}
A weaker notion is that of a \textit{well quasi-order}. They're much easier to work with, but do not code the type of graphs we are interested in. They're also not closed under as many operations as BQOs are. Note also the strong tie between the Galvin-Prikry theorem and the definition of BQO. 
\begin{definition}
	Given a BQO $Q$, we order $R= Q^{<\omega}$ via $t_1\leqq_R t_2 \iff \exists h:\text{dom}(t_1) \rightarrow \text{dom}(t_2) $ order preserving such that $t_1(n) \leqq t_2(h(n)) $. We call this order the \textit{Higman order}.
\end{definition}
\begin{fact}
	If $Q$ is a BQO, then $Q^{<\omega} $ is BQO under the Higman order. 
\end{fact}
\begin{definition}
	A \textit{tree} is a pair $(T,\leq_T)$ where $\leq_T$ is a partial order with the property that $\forall t \in T$, $\{v \in T: v \leq_T \}$ is well ordered under $\leq_T$, and $T$ has a $\leq_T$ minimal element called the \textit{root}.
\end{definition}
\begin{definition}
	Given a set $Q$, we define $\mathcal{FT}_Q$  to be the class of all pairs $(T,l)$ where $T$ is a finite tree and $l:T\rightarrow Q$. We call members of this class \textit{finite labeled trees} and refer to the function $l$ as a \textit{labeling}. 
\end{definition}
\begin{definition}
		Given a set $Q$, we define $\mathcal{T}_Q$  to be the class of all pairs $(T,l)$ where $T$ is a tree of height $\leq \omega $ and $l:T\rightarrow Q$. We call members of this class \textit{labeled trees} and refer to the function $l$ as a \textit{labeling}. 
\end{definition}
\begin{definition}
	Let $Q$ be a quasi-order. We define two quasi-orderings, $\leqq_1$ and $\leqq_m$ on $\mathcal{T}_Q$ via
	\begin{itemize}
		\item $(T_1,l_1) \leqq_1 (T_2, l_2) \iff$ $\exists f:T_1 \rightarrow T_2$, $\forall x,y \in T_1$, $x<_{T_1} y \Rightarrow f(x) <_{T_2} f(y)$ and $l_1(x) \leqq_Q l_2(f(x)) $.
		\item $(T_1,l_1) \leqq_m (T_2, l_2) \iff$ $\exists f:T_1 \rightarrow T_2$, $\forall x,y \in T_1$, $x\leq_{T_1} y \Rightarrow f(x) \leq_{T_2} f(y)$ and $l_1(x) \leqq_Q l_2(f(x)) $.
	\end{itemize} 
\end{definition}
The main distinction between $\leqq_1$ and $\leqq_m$ is $\leqq_1$ requires the function $f$ be injective. For $\leqq_m$, the function need only preserve the tree ordering.
\begin{fact}{(Laver)}
	If $Q$ is BQO, then $ \mathcal{T}_Q$ and $\mathcal{FT}_Q $ are BQO under $\leqq_1$ and $\leqq_m$. 
\end{fact}
In the next subsection, we will briefly highlight results from Lavers work on $\sigma$-scattered orders \cite{Laver1} \cite{Laver2}. Codifying objects on labeled trees is one of the most natural ways to prove an order is a BQO. For example, consider Pouzet's theorem \cite{Pouzet}. 
\subsection{Basics of $\sigma$-scattered Orders}
\begin{definition}
	Given linear orders $L_1$ and $L_2$, we say $L_1 \leqq_{emb} L_2 $ if there exists a function $f:L_1\rightarrow L_2$,  $\forall x,y \in L_1$, $x<_{L_1}y \Rightarrow f(x) <_{L_2} f(y)$. We call such a map an \textit{embedding}. If $L_1 \leqq_{emb} L_2$, we say $L_2$ \textit{embeds a copy of} $L_1$. 
\end{definition}
\begin{definition}
	Given a linear order $L$, we define its \textit{reverse order} $L^*$ to be the pair $(L,<_{L^*})$ with $x<_{L^*}y \iff x<_{L} y$. 
\end{definition}
\begin{definition}
	A linear order is called scattered if it does not embed a copy of the rationals $\QQ$.
\end{definition}
\begin{definition}
	A cardinal $\kappa $ is called \textit{regular} if ever cofinal subset has cardinality $\kappa $. 
\end{definition}
For more on the basics of cardinals, see \cite{Kunen}.
\begin{definition}
	We define $\textbf{RC}$ to be the class of regular cardinals.
\end{definition}
Note that every regular cardinal is scattered. Moreover, $\textbf{RC}$ is a BQO under $\leqq_{emb}$.
\begin{definition}
	A linear order $L$ is called $\sigma$-scattered if it can be expressed as a countable union $L= \bigcup\limits_{n=1}^\infty L_n$ such that $L_n$ are scattered. 
\end{definition}
\begin{definition}
	A $ Q$-\textit{labeled} order is a pair $(L,l)$, where $L$ is a linear order and $l:L\rightarrow Q$. We call such an $l$ a labeling.  
\end{definition}
\begin{definition}
	Given a set $Q$ and a cardinal $\kappa$, $Q^{\kappa}$ is the set of labeled ordinals $\alpha$, where $\alpha< \kappa$.  
\end{definition}
\begin{definition}
	Let $Q$ be a quasi order. Given two labeled orders $(L_1,l_1)$ and $(L_2,l_2)$, we say $(L_1,l_1) \leqq_{emb} (L_2,l_2)$ if there is an embedding $f:L_1\rightarrow L_2$ with the property that $\forall x\in L_1 $ $l_1(x) \leqq_Q l_2(f(x))$.
\end{definition}
\begin{definition}
	Given a quasi-order $Q$, we define the class $\mathcal{C}(Q)$ to be the class of $Q$-labeled $\sigma$-scattered linear orders, ordered under $\leqq_{emb}$.
\end{definition}
\begin{fact}
	Given $\alpha, \beta \in \textbf{RC}$ uncountable, there is a unique linear order $\eta_{\alpha,\beta}$ that is $\leqq_{emb}$ maximal over the class of $\sigma$-scattered orders $L$ that do not embed $\alpha^*$ or $\beta$. 
\end{fact}
 Hausdorff showed that regular cardinals served as the building blocks of the scattered linear orders. Laver showed via similar means that regular cardinals paired with the above orders serve as the building blocks for all $\sigma$-scattered orders. It is also interesting to not that $\eta_{\omega_1,\omega_1}$ is the rationals. 
\begin{definition}
	Given a quasi order $Q$, we define $Q^{+}$  to be the disjoint union $Q\cup \{\kappa:\kappa \in \textbf{RC} \}\cup \{\kappa^* : \kappa \in \textbf{RC} \}\cup \{\eta_{\alpha,\beta}: \alpha,\beta \in \textbf{RC} \} $.
\end{definition}
\begin{fact}[Laver]
	Given a better quasi order $Q$, there is a class of $\sigma$-scattered linear orders $\mathcal{H}(Q)$ and a mapping $J :\mathcal{H}(Q) \rightarrow \mathcal{T}_{Q^{+}}$ with the property that $\forall (L_1,l_1), (L_2,l_2) \in \mathcal{H}(Q)$, $ (L_1,l_1)\nleqq_{emb} (L_2,l_2) \Rightarrow J((L_1,l_1)) \nleqq_1 J((L_2,l_2))$. 
\end{fact}
\begin{fact}[Laver]
	There is a map $J: \mathcal{C}(Q) \rightarrow \mathcal{H}(Q)^{<\omega}$ such that $(L_1,l_1) \nleqq_{emb} (L_2,l_2) \Rightarrow J((L_1,l_1)) \nleqq_{\mathcal{H}(Q)^{<\omega}} J((L_2,l_2))$. 
\end{fact}
These facts imply that if $Q$ is BQO, then so is the class of $Q$-labeled $\sigma$-scattered linear orders under $\leqq_{emb}$.  
	\subsection{Graphs Generated by BQO's}
	\begin{definition}
		Given a better quasi order $Q$, we define $\vec{Q} = \{X\in Q^\omega: \forall k\in \omega, X(k) \nleqq_Q X(k+1) \} $. We often conflate $\vec{Q} $ with the shift graph $(\vec{Q},S)$.
	\end{definition}
One property these graphs have is they do not homomorphically embed the shift graph $([\omega]^\omega, S )$ as a consequence of $Q$ being BQO \cite{Peq}. Pequignot began a search for a BQO $Q$ with the property that $\chi_B (\vec{Q}) = \aleph_0$. Since we're also interested in this property, we will give it a name. 
	\begin{definition}
		If $\chi_B(\vec{Q}) \leq 3$, we call $Q$ \textit{thin}. If $Q$ is not thin, we call it \textit{thick}.
	\end{definition}
	We will later show that the name ``thick" is rather appropriate when we discuss a corollary to the colouring algorithm. 
	\begin{lemma}
		Suppose $Q_1$ and $Q_2$ are BQO and there is a map $f:Q_1 \rightarrow Q_2$ with $q\nleqq_{Q_1} p\Rightarrow f(q) \nleqq_{Q_2} f(p)$. Then $Q_2$ thin $\Rightarrow Q_1$ thin. 
	\end{lemma}
\begin{proof}
	Consider the map $\vec{f}:\vec{Q}_1 \rightarrow \vec{Q}_2  $ given by $\vec{f}(X)(k) = f(X(k))$ for all $k\in \omega$. Note, for all $k\in \omega$, $X(k) \nleqq_{Q_1} X(k+1) \Rightarrow f(X(k)) \nleqq_{Q_2} f(X(k+1)) $ so this mapping is well defined. It is also continuous as $X_n \rightarrow X$ point-wise in $\vec{Q}_1$ will imply that $\vec{f}(X_n) \rightarrow \vec{f}(X)$ point-wise. Moreover, it is easy to check that $\vec{f}$ is a factor map. It follows that $\vec{Q}_1 \preceq_B \vec{Q}_2 $.
\end{proof}
An immediate corollary to the above is the unsurprising fact that the thin property is hereditary. That is to say, if $Q$ is thin and $Q^\prime \subseteq Q$ is given the order it inherits from $Q$, then $Q^\prime$ is also thin. More interestingly, there has been historically many instances where maps of this sort are used to prove an ordering is BQO. For example, consider the previous facts of Laver from the last subsection. It also allows us to reasonably extend the notion of thin to larger BQO's that may not be sets, but rather classes.
\begin{definition}
	Given a BQO $Q$, we say $Q$ is thin if and only if $\forall  Q^\prime \in  [Q]^\omega$, $Q^\prime$ is thin.  
\end{definition}
Unless we are talking about a concrete BQO, we will always assume a thick or thin BQO is countable when we work in the abstract. 
	\section{Main Result}
	\subsection{Colouring Algorithm}
	We begin this section by proving the colouring algorithm. 
	\begin{lemma}[colouring algorithm]
		Let $Q$ be a better quasi order and $\Phi$ be a relation of arity $n$ on $Q$. $\vec{Q} $ is thin if and only if the following two properties hold
		\begin{itemize}
			\item  $ \{X \in \vec{Q}: \forall k\in \omega,\; \Phi(X(k),...,X(n+k)  ) \}$ is Borel $3$-colourable.
			\item  $ \{X \in \vec{Q}: \forall k\in \omega,\; \neg\Phi(X(k),...,X(n+k)  ) \} $ is Borel $3$-colourable.
		\end{itemize}
	\end{lemma}
	\begin{proof}
		It is clear that if $\vec{Q}$ is thin, then the above two sets are Borel three colourable as they are both proper induced subgraphs of $\vec{Q}$. For this reason, we need only show the other direction.\\
		\\
		Consider the function $f: \vec{Q} \rightarrow 2^\omega$ via
		\begin{align*}
		f(X)(k) = 1 \iff \Phi(X(k),...,X(k+n) )
		\end{align*}
		This function is continuous and is a factor map eg. $f(S(X)) = s(f(X))$. Moreover, if we let $Y = \{ x \in 2^\omega: \forall k\in \omega\; \exists m\geq k, x(m) \neq x(k)  \}$ (the set of sequences which are never eventually constant), then we see $f^{-1}(Y) \subseteq \vec{Q}$ is Borel, closed under $S$ and is Borel three colourable by lemma 2.1. By fact 2.3, there is an $A_1\subseteq f^{-1}(Y)$ that is forward recurrent and $S$ independent. $\vec{Q} \setminus f^{-1}(Y)$ is the set of all sequences $X\in \vec{Q}$ that satisfy one of the following conditions. 
		\begin{itemize}
			\item $\exists m \in \omega$, $\forall j \geq m $, $S^j(X) \in \{X \in \vec{Q}: \forall k\in \omega,\; \Phi(X(k),...,X(n+k)  ) \} $
			\item  $\exists m \in \omega$, $\forall j \geq m $, $S^j(X) \in \{X \in \vec{Q}: \forall k\in \omega,\; \Phi(X(k),...,X(n+k)  ) \} $
		\end{itemize}
		By our hypothesis, $ \{X \in \vec{Q}: \forall k\in \omega\; \Phi(X(k),...,X(n+k)  ) \}$ and $ \{X \in \vec{Q}: \forall k\in \omega \;\neg\Phi(X(k),...,X(n+k)  ) \} $ admit forward recurrent $S$ independent sets $A_2$ and $A_3$ respectively. Since every $X \in \vec{Q}$ is either in $f^{-1}(Y)$ or eventually in $A_2$ or $A_3$, $A_1\cup A_2 \cup A_3$ is a witness to $\chi_B(\vec{Q})\leq 3$ as it is a Borel $S$ independent forward recurrent set. Note, independence is a consequence of each $A_i$ belonging to disjoint $S$ closed sets.   
	\end{proof}
Here are two simple applications of the lemma that will be of use later. 
\begin{lemma}
	If $ Q_1$ and $Q_2$ are thin, then the disjoint union $R=Q_1\cup Q_2$ is thin.
\end{lemma}
\begin{proof}
	Consider the unary relation $\Phi$ on $R$ given by $\Phi(q) \iff q \in Q_1$. By our colouring algorithm, it suffices to consider whether or not the subgraphs induced by the following subsets are Borel $3$-colourable.
	\begin{align*}
		B_1 &= \{X\in R: \forall k \in \omega, X(k) \in Q_1\}\\
		B_2 &= \{X\in R: \forall k \in \omega, X(k) \in Q_2 \}
	\end{align*}
	Notice however, that the first is simply $\vec{Q}_1$ and the second is $\vec{Q}_2$. As both $Q_1$ and $Q_2$ are thin, we are done. 
\end{proof}
Note, this implies that finite disjoint unions, and potentially non-disjoint unions, of thin BQOs are thin as well.
\begin{lemma}
	If $Q_1$ and $Q_2$ are thin, then $R= Q_1\times Q_2$ is thin.
\end{lemma}
\begin{proof}
	Consider the binary relation $\Phi$ on $R$ given by $\Phi((p_1,p_2), (q_1,q_2) ) \iff p_1 \leqq_{Q_1} q_1 $. Let $\pi_1: R\rightarrow Q_1 $ and $\pi_2: R\rightarrow Q_2$ be canonical projection maps ($\pi_i(p_1,p_2)=p_i$). By our colouring algorithm, we need only consider the colourability of the subgraphs generated by the sets
	\begin{align*}
		B_1 &= \{X\in \vec{R}: \forall k \in \omega \; \Phi(X(k), X(k+1) )  \}\\
		B_2 &= \{X\in \vec{R}: \forall k \in \omega \; \neg \Phi(X(k), X(k+1) )  \}
	\end{align*}
	Notice that if $(p_1,q_1) \nleqq_{R} (p_2,q_2) $, then either $p_1 \nleqq_{Q_1} q_1$ or $p_2\nleqq_{Q_2} q_2$. Consequently,  $\neg \Phi((p_1,p_2), (q_1,q_2) )$ and $(p_1,p_2) \nleqq_R (q_1,q_2) $  $\Rightarrow$ $p_2 \nleqq_{Q_2} q_2$. It follows that for $i\in \{1,2\} $ the mappings $\vec{\pi}_i: B_i \rightarrow \vec{Q}_i$ given by $\forall k \in \omega $ $ \vec{\pi}_i (X)(k) = \pi_i(X(k))$ is a well defined continuous factor map. Consequently, they are Borel graph homomorphisms and since $Q_1$ and $Q_2$ are thin, we are done. 
\end{proof}
We can also use the colouring algorithm to deduce a property of thick graphs that suggests the name is quite appropriate. First, note that being thick means that the relation $\nleqq_Q$ must be complex. Consequently, given a binary relation $\Phi\subseteq Q^2$, it is possible for the BQO $\leqq_Q \cup \Phi$ to no longer be thick. For example, if $\Phi = Q^2$. However, given the choice between $\Phi$ and its compliment $\neg \Phi$, one of $\leqq_Q \cup \Phi$ or $\leqq_Q \cup (\neg \Phi)$ must be thick. 
\begin{prop}
	Let $Q$ be thick. Let $\mathcal{I} = \{\Phi \subseteq Q^2 : \leqq_Q \cup \Phi \; \text{thick}  \}$. Then $\mathcal{I}$ satisfies the following properties:
	\begin{itemize}
		\item $\Psi \subseteq \Phi$, $\Phi\in \mathcal{I}$ $\Rightarrow $ $\Psi \in \mathcal{I}$.
		\item $Q^2 \notin \mathcal{I}$.
		\item $\emptyset \in \mathcal{I}$.
		\item $\Phi \notin \mathcal{I} \Rightarrow \neg \Phi \in \mathcal{I}$.  
	\end{itemize}
\end{prop}
This suggests there is a link between thick BQOs and directed sets (eg. filters and ideals). 
	\subsection{From $Q$ to $Q^{<\omega}$}
	For this section, we will fix $Q$ to be a BQO and $R$ to be $Q^{<\omega}$ under Higman's order. We also let $B = \{X\in \vec{R} : \forall k \in \omega, \text{len}(X(k)) \leq \text{len}(X(k+1)) \} $. 
	\begin{definition}
		Given $X\in B$, we let $m_k(X) \in \omega$ be the largest integer such that $X(k) \upharpoonright m_k(X) \leqq_R X(k+1)$. We let $n_k(X)$ be the smallest integer such that $X(k) \upharpoonright m_k(X) \leqq_R X(k+1) \upharpoonright n_k(X)$.        
	\end{definition}
	Note, for any $X$ as in the above, $m_k(X) \leq n_k (X)$. Also, if $X_i \in B $ is a sequence that converges to $X$, then $\forall k \in \omega $ $n_k(X_i) \rightarrow n_k(X)$ and $m_k(X_i) \rightarrow m_k(X) $. Thus, both $n_k$ and $m_k$ are continuous for any $k$. They also satisfy $n_k(S(X)) = n_{k+1}(X)$ and $m_k(S(X))= m_{k+1}(X)$.
	\begin{definition}
		Given $X\in B$, if $\forall k \in \omega$, $m_k(X) \leq m_{k+1}(X) < n_k(X) \leq n_{k+1}(X)$, then we define $\partial X \in B $ by $\forall k \in \omega$, $\partial X(k) = X(k)\upharpoonright m_k$. This is well defined since $n_k$ was minimal and $m_{k+1}<n_k$. We call such $X$ \textit{$1$-derivable}. We say $\partial^0 X = X$.  Given $i\in \omega$, we recursively define $\partial^{i} X = \partial (\partial^{i-1} X)$ when $\partial^{i-1} X$ is $1$-derivable. We call such $X$ \textit{$i$-derivable}        
	\end{definition}
	Derivations commute with $S$ i.e $\partial^i S(X) = S(\partial^i X)$. This is a consequence of the previously noted fact $n_k(S(X)) = n_{k+1}(X)$ and $m_k(S(X))= m_{k+1}(X)$. It follows from $\partial$ commuting with $S$ that if $X$ is $i$-derivable, then $S(X)$ is $i$-derivable. 
	\begin{definition}
		We let $D\subseteq B $ be the set of all $1$-derivable members from $B$. 
	\end{definition}
The mapping $\partial : D \rightarrow \vec{R}$ is a continuous factor map. 
	\begin{prop}
		$\forall X\in B$, there is a maximal $i$ such that $X$ is $i$-derivable. 
	\end{prop}
	\begin{proof}
		Suppose otherwise. Note $\text{len}((\partial^i X) (0) ) = m_0( \partial^{i-1} X )$ decreases in length. We then have an infinite decreasing sequence in $\omega$ which is impossible. 
	\end{proof} 
	\begin{definition}
		For an $X\in B$, let $M_k(X) \in \omega $ be the largest such ordinal such that $S^k(X)$ is $M_k(X)$-derivable. We define $\partial^{\infty} X $ by $\partial^\infty X(k) = (S^k(X))^{(M_k)} (0)$.
	\end{definition}
	Note that for every $k\in \omega $, the mappings $M_k(X) $ and $\partial^{M_k(X)} X$ are continuous.  
	\begin{lemma}
		The mapping $X\in B$ defined by $X\rightarrow \partial^\infty X$ is well defined and Baire class 1. Moreover, $\forall k \in \omega$ $S^k(\partial^\infty X)\in B\setminus D $
	\end{lemma}
	\begin{proof}
		First, since $X$ $i$-derivable $\Rightarrow $ $S(X)$ is $i$-derivable, $M_k(X)$ is an increasing sequence. Moreover, $S^{k}(\partial^{M_{k-1}(X)} X)$ is an initial segment of $S^{k}(\partial^{M_{k}(X) } X) $. Consequently, $\partial^\infty X(k) = (S^{(k)}(\partial^{M_k(X)}X))(0)\nleqq_R (S^{(k+1)}(\partial^{M_k(X)}X))(0)$ implies that $\partial^\infty X (k) \nleqq_R \partial^\infty X (k+1)$ as $\partial^\infty X (k+1) $ is an initial segment of $(S^{(k+1)}(\partial^{M_k(X)} X))(0)$. \\
		\\
		To see the mapping $X\rightarrow \partial^\infty X$ is Baire class 1, consider the mappings $f_j :\vec{R}\rightarrow \vec{R}$ via
		\begin{align*}
		\forall i<j, f_j(X)(i) &= (S^i(\partial^{M_i(X) } X)) (0)\\
		f_j(X)(i)&= (S^i(\partial^{M_j(X) } X)) (0) \; \text{otherwise}
		\end{align*}
		It is clear that these maps are continuous as finite derivations are continuous. It is also clear that $f_k(X) \rightarrow \partial^\infty X$ point-wise.  \\
		\\
		Note that $S(\partial^{\infty } X)= \partial^{\infty }S(X)  $ as derivation commutes with $S$. The existence of a $k$ such that $S^k(\partial^{\infty } X ) \in D$  would contradict the maximality of $M_k(X)$ as we could have taken one further derivation. 
	\end{proof}
	\begin{lemma}
		Given an $X\in B$, if $m_{k+1}(X) \geq n_k(X)$, the sequence $Y(k) = X(k)(m_k(X)+1) $ is in $\vec{Q}$.
	\end{lemma}
	\begin{proof}
		Note, for all $d\geq n_k(X)$, $X(k)(m_k(X)+1) \nleqq_Q X(k+1)(d)$. Else, the existence of such a $d$ would mean $X(k)\upharpoonright m_k(X)+1 \leqq_R X(k+1) \upharpoonright  d \leqq_R X(k+1)$. Since $m_k(X) \leq n_k(X) \leq  m_{k+1}(X)$, $\forall k \in \omega $ $ X(m_k(X)+1) \nleqq_Q X(m_{k+1}(X) +1) $ as desired. 
	\end{proof} 
	\begin{theorem}
		If $Q$ is thin, then so is $Q^{<\omega}$.
	\end{theorem}
\begin{proof}
	By our colouring algorithm, it suffices to Borel $3$-colour the set $B$. To see this, consider the relation $\Phi(r,t) \iff \text{len}(r) \leq \text{len}(t)$. Since there is no infinite descending sequence in $\omega$, we only need to worry about the $\Phi$ homogeneous members in $\vec{R}$ which is the set $B$. By lemma 3.4 and lemma 2.1, the mapping $\partial^\infty : D\rightarrow B\setminus D$ is a Borel graph homomorphism. Hence, it suffices to consider $B\setminus D$. Applying our algorithm twice, we can reduce colourability to the set of $X\in B\setminus D$ that have the parameters $m_k(X) $ and $n_k(X)$ increasing.
	\begin{align*}
	C &= \{X\in B\setminus D: \forall k \in \omega, m_k(X) \leq m_{k+1}(X), n_k(X) \leq n_{k+1}(X)  \}
	\end{align*} 
	Again using the algorithm, we can reduce to the case where $\forall k \in \omega$ $m_{k+1}(X) \geq n_k(X) $ or $\forall k \in \omega$, $m_{k+1}(X) <n_k(X)$. Note, the latter is impossible for members of $C$, else such an $X$ would be in $D$. By lemma 3.5, $\{X\in C: m_{k+1}(X) \geq n_k(X) \}\preceq_B \vec{Q}$ by the mapping $f(X)(k) = X(m_k(X) +1)$. Since $Q$ is thin, we have Borel $3$-colouring of $\vec{R}$. Thus, $Q$ thin implies $Q^{<\omega}$ thin.
\end{proof} 
	\begin{theorem}
		If $Q$ is thin, then so is $\mathcal{FT}_Q$ under $\leqq_1$.
	\end{theorem}
	\begin{proof}
		For each $T\in \mathcal{FT}$, let $\hat{T}$ be a linear order extending the order of $T$. Given $(T,l) \in \mathcal{FT}_Q$, we can identify the pair $(\hat{T},l)$ with a sequence in $Q^{<\omega}$. This is because of course, every finite linear order is uniquely determined by size and a $Q$ labeled linear order is then just a $Q$ sequence. Notice that this mapping $(T,l) \rightarrow (\hat{T} , l ) $ satisfies $(S,m) \leqq_1 (T,l) \Rightarrow (\hat{S},m) \leqq_{Q^{<\omega}}  (\hat{T},l)  $. This naturally gives us a homomorphism from $\vec{\mathcal{FT}_Q}$ into $\vec{Q^{<\omega}} $ via lemma 2.2.
	\end{proof}
	\begin{corollary}
		If $Q$ is thin, so are $[Q]^{<\omega}$ and $\mathcal{FT}_Q$ under $\leqq_m $.
	\end{corollary}
	\begin{proof}
		First, well order $Q$. Then, given an $a \in [Q]^{<\omega} $, assign to it the sequence $\hat{a} \in Q^{<\omega} $ of its members placed in order according to the well order. Note that $\hat{a} \leqq_{Q^{\omega}} \hat{b} \Rightarrow a \leqq_m b$. Consequently, $[Q]^{<\omega} $ is thin by theorem 3.6 and lemma 2.2. Similarly for trees, $(S,m)\leqq_1 (T,l) \Rightarrow  (S,m) \leqq_m (T,l)$ for $A,B \in  [Q]^{<\omega}$ and $(S,m),(T,l) \in \mathcal{FT}_Q $. So, by lemma 2.2 and theorem 3.7 $\mathcal{FT}_Q$ is thin under $\leqq_m$.
	\end{proof}
\subsection{Finite To Infinite}
Since Laver showed that labeled $\sigma$-scattered orders are in some sense, no more complicated than finite labeled trees, it comes as no surprise that if $Q$ is thin, so are $Q$-labeled $\sigma$-scattered orders. We will prove this here. We start by showing that any countable collection of $Q$ labeled ordinals where $Q$ is thin, is also thin under $\leqq_{emb}$.   
\begin{lemma}
	Let $Q$ be thin. Then for any cardinal $\kappa$, $Q^{ <\kappa}$ is thin under $\leqq_{emb}$
\end{lemma}
\begin{proof}
	Let $R$ be a countable subset of $Q^{<\kappa}$. Applying our colouring algorithm to $\vec{R}$, we need only consider 
	 \begin{align*}
		A &= \{X \in \vec{R} : \text{len}(X(k))\leq \text{len}(X(k+1))  \}
	 \end{align*}
	 Akin to in the finite case, for $X \in A$ and $k\in \omega$, there is an $\alpha_k(X), \beta_k(X) <\kappa$ where $\alpha_k(X)$ is minimal such that $X(k) \upharpoonright \{\beta : \beta< \alpha_k(X) \} \nleqq_{emb} X(k+1)$ and $\beta_k(X)$ is minimal such that $X(k) \upharpoonright \alpha_k(X) \leqq_{emb} X(k+1)\upharpoonright \beta_k(X)$. Note when $\kappa = \omega$, this is identical to the finite case. Since $R$ is a countable subset, and members from $R$ may be uncountable labeled ordinals, it may not be the case that for every $(\gamma, l) \in R $ and $\zeta< \gamma$, $(\zeta, l\upharpoonright \zeta) \in R$. However, we may assume that $R$ is closed under all relevant restrictions. First, we will prove this claim.
	 \\
	 \\
	 Given a pair $(\gamma_1, l_1 ), (\gamma_2,l_2) \in Q^{<\kappa}$, we define $\alpha( (\gamma_1, l_1 ), (\gamma_2,l_2) )  $ to be the minimal $\alpha$ such that $(\alpha, l_1 \upharpoonright \alpha )\leqq_{emb} (\gamma_2,l_2) $ and $\beta((\gamma_1, l_1 ), (\gamma_2,l_2)) $ the minimal $\beta$ such that $(\alpha,l_1 \upharpoonright \alpha )\leqq_{emb}  (\beta,l_2 \upharpoonright\beta)$. Given $R\subseteq Q^{<\kappa} $, we define the following sets
	 \begin{align*}
	 	R_l &= \{ (\alpha, l): \exists (\gamma_1, l_1 ), (\gamma_2,l_2) \in R, \; (\alpha, l) = (\alpha((\gamma_1, l_1 ), (\gamma_2,l_2)), l_1 \upharpoonright \alpha((\gamma_1, l_1 ), (\gamma_2,l_2)) )   \}\\
	 	R_r &= \{ (\alpha, l): \exists (\gamma_1, l_1 ), (\gamma_2,l_2) \in R, \; (\alpha, l) = (\alpha((\gamma_1, l_1 ), (\gamma_2,l_2)), l_2 \upharpoonright \alpha((\gamma_1, l_1 ), (\gamma_2,l_2)) )   \} \\
	 	R^\prime &= R_l\cup R_r
	 \end{align*}
	 Note that if $R$ is countable, so is $R^\prime$ as $\alpha(\cdot,\cdot)$ is parameterized over pairs from $R$ of which there are only countably many. Given an $R$, we define recursively $\forall i \in \omega$, $R^0 = R$, $R^i = (R^{i-1})^\prime$. The set $R_{fill} = \bigcup\limits_{i=0}^\omega R^i $ contains $R$ is a subset and has the property that for any $ (\gamma_1, l_1 ), (\gamma_2,l_2) \in R_{fill} $, $(\alpha((\gamma_1, l_1 ), (\gamma_2,l_2)), l_1 \upharpoonright \alpha((\gamma_1, l_1 )$ and $(\alpha((\gamma_1, l_1 ), (\gamma_2,l_2)), l_2 \upharpoonright \alpha((\gamma_1, l_1 ) $ are in $R_{fill}$. Consequently, we could define a derivative operation as we did in the finite case on $\vec{R_{fill}}  \supseteq \vec{R}$. We will explain why this is the case next, but for now we've at least shown that we can assume $R$ is sufficiently closed under restrictions. 
	 \\
	 \\
	 Note that it is never the case that $\alpha_{k}(X) = \text{dom}(X(k))$. If it were, then for each $\beta \in \text{dom}(X(k))$, there is an $h(\beta) $ such that $X(k)(\beta) \leqq_Q X(k+1)(h(\beta))$ and $h$ strictly increasing. But of course, this $h$ codifies that $X(k) \leqq_R X(k+1)$, a contradiction. It is also the case that $\forall \gamma \geq \beta_k(X)$, $X(k)(\alpha_k(X)) \nleqq_Q X(k+1)(\gamma)$ akin to the finite case. This means that the sequence $X(k)(\alpha_k(X)) \in \vec{Q}$ when $\alpha_k(X) \leq \beta_k(X) < \alpha_{k+1}(X)$.  
	 \\
	 \\
	 Similar to before, we can say an $X$ is derivable if $\forall k \in \omega$, $\alpha_k(X) \leq \alpha_{k+1}(X)  \leq \beta_k(X) \leq \beta_{k+1}(X)$. We can also define the derivation operation $\partial$ and it's iterates by restricting to $\alpha_k(X)$ coordinate-wise as before. Since we can assume $R$ is closed under these coordinate-wise restrictions, the derivative operation is well defined on $\vec{R}$. Moreover, $\partial$ is a Borel graph homomorphisms, and every $X$ is at most $i$ derivable for some $i$ as there is no infinite descending sequence of ordinals. Consequently, we can define the operator $\partial^\infty$ for derivable members of $\vec{R}$. This mapping is also well defined. It is also a Borel factor map and mimicking the argument from the finite case, we have that $\vec{R}$ is thin. 
\end{proof}
\begin{theorem}
	If $Q$ is thin, then so is $\mathcal{T}_Q$ under $\leqq_1$.
\end{theorem}
\begin{proof}
	Let $R$ be a countable subset of $\mathcal{T}_Q$. Let $\mathcal{I} = \{T:(T,l) \in R \}$. Note that each tree $T\in \mathcal{I}$ can have its order extended to a well order. Simply well order each set of immediate successors from $T$ and then give $T$ the lexicographical order. Given a $T\in \mathcal{I}$, we call the well order $\hat{T}$. Since $\mathcal{I}$ is countable, there is a cardinal $\kappa$ such that each $(\hat{T},l) \in Q^{<\kappa}$. Since the ordering extends the tree order, we have $(T_1,l_1) \nleqq_R (T_2, l_2) \Rightarrow (\hat{T}_1,l_1) \nleqq_{Q^{<\kappa}} (\hat{T}_2, l_2) $. It follows that $\vec{R}$ is thin by our last lemma. Hence, $ \mathcal{T}_Q$ is thin.  
\end{proof}
\begin{corollary}
	If $Q$ is thin, then the class of $Q$ labeled $\sigma$-scattered linear orders is thin. 
\end{corollary}
\begin{proof}
	It suffices to show $Q^+$ is thin. Since the disjoint union of thin sets are thin, it suffices to show that $\textbf{RC}$ and $\textbf{RC}\times \textbf{RC} $ are thin. Since products of thin sets are thin, we need only show that $\textbf{RC}$ is thin. Given any $R\subseteq \textbf{RC}$ countable, $\vec{R}$ is empty as there is no infinite descending sequence of ordinals. Hence, $\textbf{RC}$ is thin as desired. 
\end{proof}
\section{Acknowledgements}
The author would like to thank Yann Pequignot for their email correspondence, as well as for sharing some notes on the problem. The author would also like to thank the Calgary discrete mathematics seminar for allowing the author to present his initial findings. Finally, the author would like to thank Stevo Todorcevic for introducing him to the problem, and supporting him through the solution process. 
	\newpage

\bibliographystyle{ijmart}

\end{document}